\documentclass[12pt]{amsart}

\usepackage{amsthm}
\usepackage{amsmath}
\usepackage{amsfonts,amssymb,latexsym}
\usepackage[english]{babel}
\usepackage{epsfig,epsf}
\usepackage{color}
\newtheorem{thm}{Theorem}
\newtheorem{lem}{Lemma}

\newtheorem{theirtheorem}{Matolcsi-Ruzsa Theorem}
\renewcommand{\thetheirtheorem}

\newtheorem{theirconj}{Chung-Goldwasser-Matolcsi-Ruzsa Conjecture}
\renewcommand{\thetheirconj}

\newcommand{\Acal}{{\mathcal A}}

\newcommand{\R}{{\mathbb R}}

\newcommand{\veps}{\varepsilon}

\title[Maximal sets with no solution to $x+y=3z$]{Maximal sets with no solution to $x+y=3z$}
\author[Alain Plagne and Anne de Roton]{Alain Plagne \quad and \quad Anne de Roton}
\thanks{Both authors are supported by the ANR grant C\ae sar.}
\email{plagne@math.polytechnique.fr}
\email{anne.de-roton@univ-lorraine.fr}
\address{CMLS, \' Ecole polytechnique, 91128 Palaiseau Cedex, France}
\address{IECN, Universit\' e de Lorraine, B.P. 239, 54506 Vandoeuvre-l\` es-Nancy Cedex, France}

\begin{document}

\begin{abstract}
In this paper, we are interested in a generalization of the notion of sum-free sets. We address a conjecture first made 
in the 90s by Chung and Goldwasser. Recently, after some computer checks, this conjecture was formulated again by Matolcsi and Ruzsa, who 
made a first significant step towards it. Here, we prove the full conjecture by giving an optimal upper bound for the Lebesgue measure 
of a 3-sum-free subset $A$ of $[0,1]$, that is, a set containing no solution to the equation $x+y=3z$ where $x,y$ and $z$ are restricted to belong 
to $A$. We then address the inverse problem and characterize precisely, among all sets with that property, those attaining the 
maximal possible measure.
\end{abstract}

\maketitle

\section{Introduction}

The problem of {\em sum-free sets} or more generally of {\em $k$-sum-free sets} ($k$ is a positive integer) has a long history (see for instance \cite{WSW}). 
A  subset of a given (additively written) semi-group, say, is said to be $k$-sum-free if it contains no triple $(x,y,z)$ satisfying the equation $x+y=kz$. 
In the case of integers, which was certainly the first historical case of study, Ruzsa  \cite{Ru1, Ru2} studied more general linear 
equations and introduced a new terminology by distinguishing between what he called {\em invariant} and {\em noninvariant} equations. 
Invariant equations, which correspond here to the fact that the sum of the coefficients of the unknowns in the forbidden relation is equal to zero,
lead to the existence of trivial solutions -- as appears for instance in the case of 2-sum-free sets (since $x+x$ is equal to $2x$, whatever $x$ is) --
which have not to be considered and lead to special developments: 2-sum-free sets, which are also and in fact mainly known 
as sets without arithmetic progressions of length $3$, are of great importance and their study is central in additive combinatorics. 
We simply mention \cite{Sa} for the latest development on the subject which goes back at least to
 Roth \cite{Roth}. 
In the present paper, we shall only deal with the non invariant cases, that is, $k$ is supposed different from $2$. In this case, 
the problems which appear are of a different kind.

The very basic question to maximize the cardinality of a set of integers included in $\{ 1, 2,\dots, n \}$ having no 
solution to the equation $x+y=z$ (sum-free sets) belongs to the folklore and is easily solved (see for instance \cite{CG2} or \cite{DFST}).
One cannot select more than $\left\lceil n/2 \right\rceil$ integers with the required property, and this is optimal. 
Interestingly, for a general $n$, there are two kinds of extremal sum-free sets (see Theorem 1.1 of \cite{CG2} for a precise statement): 
the {\em combinatorial} one, namely the upper-half, $\{ \lceil (n+1)/2 \rceil , \dots, n \}$ for which the impossibility to solve the 
equation follows from a size condition; and the {\em arithmetic} one, in the present case the set of odd integers, for which a modular 
condition prevents from the existence of a solution. Not only in the case of sum-free sets of integers is this dichotomy emerging.
In all these types of questions, when asked in a discrete setting, this typology is subject to appear.

For $k=3$ (and $n\not=4$), Chung and Goldwasser \cite{CG2} proved Erd\H{o}s' conjecture that $\lceil n/2\rceil$ is the maximal 
size of a 3-sum-free set of positive integers less than $n$. They also prove, at least when $n \geq 23$ (see Theorem 1.3 in \cite{CG2}),  
that the set of odd integers is the only example attaining this cardinality. 

For $k \geq 4$, Chung and Goldwasser \cite{CG} discovered $k$-sum-free subsets of $\{ 1, 2,\dots, n \}$ with a size asymptotic to
$$
\sim \frac{k-2}{k^2 -2} \left( k + \frac{8}{k(k^4-2k^2 -4)} \right) n
$$
as $n$ tends to infinity. This was obtained thanks to an explicit construction of three intervals of integers. They additionally conjectured 
that this was the actual exact asymptotic maximal value. This conjecture was finally settled by Baltz, Hegarty, Knape, Larsson and Schoen 
in \cite{BHKLS}. These authors additionally proved an inverse theorem giving the structure of a $k$-sum-free sets of this size : such sets 
have to be close from the set composed of the three above-mentioned intervals.

In fact, Chung and Golwasser  managed to predict the maximal size of a $k$-sum-free set of integers less than $n$ by studying 
the continuous analog of the problem in \cite{CG}; in other words by introducing the study of $k$-sum-free subsets of reals number selected 
from $[0,1]$.  Indeed, a $k$-sum-free subset of $[0,1]$ leads, after a suitable dilation, to a $k$-sum-free set of integers 
(but it is important to notice, this set will be mandatorily -- in the typology mentioned above --  of a combinatorial nature). 

We thus arrive to the question of determining the maximal Lebesgue measure -- denoted thereafter $\mu$ -- of a subset of $[0,1]$ 
having no solution to the equation $x+y=kz$. The case $k=1$ is easy and, as mentioned above, the cases $k \geq 4$ were solved 
in \cite{CG}. However, the case $k=3$ was left open and remained the only one for which the optimal asymptotic density was unknown. 
Nonetheless, it was precisely investigated and the set (composed again of three intervals)
\begin{equation}
\label{ensA}
\Acal_0 = \left( \frac{8}{177}, \frac{4}{59} \right) \cup \left( \frac{28}{177}, \frac{14}{59} \right) \cup \left( \frac{2}{3},1 \right)
\end{equation}  
which does not contain a solution to the equation $x+y=3z$, was identified in \cite{CG} as playing an important role in the question. Notice that
its measure is equal to $77/177 = 0.4350\dots$ In the sequel, we shall call $\Acal_1,\dots, \Acal_7$ the seven sets 
defined as the union of $\Acal_0$ and three points, one end-point of each interval appearing in the definition of $\Acal_0$, 
except  $\{8/177,14/59,2/3\}$.  These seven sets are 3-sum-free.
The quite precise following  conjecture was then formulated in \cite{CG} (and proposed again later in \cite{MR} after 
several computer-aided checks):

\begin{theirconj}
\label{conjCG}
Let $A$ be a measurable 3-sum-free subset of $[0,1]$. 
Then
$$
\mu (A) \leq \frac{77}{177}.
$$ 
Moreover, if $\mu(A)=77/177$ and if $A$ is maximal with respect to inclusion among the 3-sum-free subsets of $[0,1]$, 
then $A \in \{ \Acal_1, \dots, \Acal_7 \}$.
\end{theirconj}

Recently, Matolcsi and Ruzsa  \cite{MR} made the first breakthrough towards the first part of this conjecture by showing 
the following theorem.

\begin{theirtheorem} 
Let $A$ be a mesurable 3-sum-free subset of $[0,1]$. Then its measure satisfies 
$$
\mu (A) \leq \frac{28}{57} = 0.49122\dots
$$
\end{theirtheorem}

This result, the first one to prove a strictly less than $0.5$ upper bound for 3-sum-free subsets of $[0,1]$,  
is very noticeable because it shows in particular that in the case of $3$-sum-free sets, 
contrary to what happens in the other cases, the maximal size of such a subset in $[0,1]$ is not the 
analog of that of a $k$-sum-free subset of integers  (let us recall that such a set has a density $1/2$). 
This illustrates indeed the fact that the only known $3$-sum-free set of integers of maximal size is 
the set of odd numbers, a set of an arithmetic nature which does not possess a continuous analog, 
contrarily to sets of combinatorial nature. This is an important observation : in the case of the analogy 
of cyclic groups of prime order and the torus, a recent theorem of Candela and Sisask (see Theorem 1.3 in \cite{CS}) 
shows that the discrete model always converges towards the continuous one. Notice that a good reason 
for this to happen, in this discrete case, is that  
even sets of an arithmetic nature can be transformed without loss of generality in sets of a combinatorial nature 
with the same density: multiplying an arithmetic progression by the inverse of its difference transforms it into an interval.
This does not happen in the case of present study and makes the behaviour of maximal sets more difficult to handle.

In this paper, we first  establish the optimal (in view of  example \eqref{ensA}) upper bound for the measure 
of 3-sum-free sets of $[0,1]$.

 \begin{thm}
 \label{main}
 Let $A$ be a measurable 3-sum-free subset of $[0,1]$.
 Then  
 $$
 \mu (A) \leq \frac{77}{177}.
 $$ 
 \end{thm}
 
 Then we solve the inverse associated problem.
 
\begin{thm}
\label{thm2}
 Let $A$ be a measurable 3-sum-free subset of $[0,1]$ satisfying $\mu (A) = 77/177$, then there is an $i \in \{ 1,\dots, 7\}$ such that $A \subset \Acal_i$.
\end{thm}
 
The full Chung-Goldwasser-Matolcsi-Ruzsa conjecture is thus proved.

\section{Notations and prerequisites}

In what follows, we denote respectively  by $\mu (X)$ and $\rm{diam} (X) = \sup X - \inf X$, the Lebesgue measure and the diameter
of a set $X$ of real numbers. We shall denote by $A+B$ the Minkowski sumset of two subsets $A$ and $B$ of $\R$, and 
by $\alpha \cdot A$ the $\alpha$-dilate of $A$, that is  $ \{ \alpha x \text{ for } x \in A \}$. Notice in particular that $2 \cdot A$ is included in, 
but in general different from, $A+A$. 

While the behaviour of $\mu$ with respect to dilation is clear since one has 
\begin{equation}
\label{dilatation}
\mu (\alpha \cdot A) = \alpha \mu ( A),
\end{equation}
it is more complicated 
for the case of Minkowski addition. The basic estimate for the measure of the sum of two measurable bounded subsets $A$ and $B$ of $\R$  
is a standard Brunn-Minkowski type \cite{HM} lower bound, namely
\begin{equation}
\label{BM}
\mu_* (A+B )\geq  \mu (A)+\mu (B),
\end{equation}
where $\mu_*$ denotes the inner measure (the use of this tool is made necessary by the fact that $A+B$ is not necessarily 
measurable as shown by Sierpi\' nski \cite{Sierpi}).

Beyond this, the best known result is due to Ruzsa.

\begin{lem} (Ruzsa \cite{Ruz})
\label{ruzsa}
Let $A$ and $B$ be two bounded measurable subsets of $\R$ such that $\mu (A)\leq\mu (B)$, then  
\begin{equation}
\label{a+b}
{\mu_* } (A+B )\geq  \min(2\mu (A)+\mu (B), \mu (A)+{\rm diam}(B)).
\end{equation}
In particular, one has
\begin{equation}
\label{a+a}
{\mu_* } (A+A )\geq  \min(3\mu (A), \mu (A)+{\rm diam}(A)).
\end{equation}
\end{lem}

We underline the fact that up to this lemma, the result presented in this paper is self-contained.

We now state two more specific lemmas, due to  Matolcsi and Ruzsa \cite{MR}, 
that we shall need in the present study.
These intermediary results are not presented as lemmas in \cite{MR}, therefore, to ease 
the reading of the present paper, we include their respective (condensed) proofs here.
Before entering this, we note that the assumption that there is no solution to 
$x+y=3z$ with $x,y,z \in A$ can be rewritten set-theoretically in the form
\begin{equation}
\label{hypothese}
(A+A) \cap (3\cdot A) = \emptyset \quad \text{ or, equivalently, } \quad \left( \frac{1}{3} \cdot (A+A) \right) \cap A = \emptyset.
\end{equation}

\begin{lem} (Matolcsi-Ruzsa \cite{MR})
\label{majreste}
Let $A$ be a measurable bounded 3-sum-free subset of $\R^+$. 
One has
$$
\mu (A) \leq \frac{2 \sup A - \inf A}{4}.
$$
\end{lem}

\begin{proof}
Since, by definition, $1/3 \cdot (A+A)$ and $A$ are intersection-free and both included 
in the interval $[ 2 \inf A /3 , \sup A ]$, we obtain
\begin{eqnarray*}
\sup A - \frac23 \inf A & \geq & \mu_*  \left(  \frac13 \cdot (A+A)  \right) + \mu (A)  \\
			   & \geq & \min \left( \mu (A), \frac13 ( \mu (A)+ \sup A - \inf A) \right) + \mu (A),
\end{eqnarray*}
in view of \eqref{dilatation} and \eqref{a+a}. If $2 \mu (A) \geq {\rm diam}(A)$, then we obtain
$$
\sup A - \frac23 \inf A  \geq \frac43  \mu (A)+ \frac13 (\sup A - \inf A)
$$
which gives the result. In the other case, we 
have
$$ \mu (A) \leq \frac12{\rm diam}(A)=\frac12(\sup{A}-\inf{A})\leq\frac{2\sup{A}-\inf{A}}{4},$$
since $A \subset \R^+$.
\end{proof}

Here is the second lemma useful to our purpose.

\begin{lem}(Matolcsi-Ruzsa \cite{MR})
\label{lemMR}
Let $A$ be a measurable 3-sum-free subset of $[0,1]$ such that $\sup{A}=1$, then
$$
\mu (A) \leq \frac13 + \frac12 \mu \left( A \cap \left[\frac23,1\right] \right).
$$
\end{lem}

\begin{proof}
We define
$$
a =\inf{A}, \quad  A_{1}=A\cap\left[\frac23,1\right] \quad \text{ and }\quad \veps=\frac13-\mu(A_1)
$$ 
and refine the argument used in the proof of Lemma \ref{majreste} using that the three sets $1/3 \cdot (A+A)$, $A$ and $[2/3,1] \setminus A_1$
are disjoint in $[2a/3,1]$. This gives
$$
\mu (A) \leq \frac{1- \veps}{2} - \frac{a}{4}\leq \frac{1- \veps}{2}= \frac13 + \frac12 \mu (A_1),
$$
that is, the result.
\end{proof}

\section{Two central lemmas}

The proof of the main theorem (Theorem \ref{main}) relies essentially on the following technical lemma.

\begin{lem}
\label{lem0}
Let $A$ be a measurable 3-sum-free subset of $[0,1]$ 
such that $\sup{A}=1$. Let
$$
a =\inf{A}  \quad  {\text  and } \quad A_{1}=A\cap\left[\frac23,1\right]
$$ 
and define
$$
\veps_1 = \inf{A_{1}} - \frac23 \quad  \text{ and } \quad \veps_2= \left( \frac13-\veps_1 \right)- \mu ( A_1 ).
$$

If $\veps_1+2\veps_2\leq 1/3$, then one has
\begin{equation*}
\mu \left( A \cap \left[\frac29+\frac{a}{3},1\right]\right)  \leq
\begin{cases}
\frac13-\frac16\veps_1&\text{ if }\veps_1\leq \frac23a,\\ 
\frac13 -\frac{1}{24} \left( \veps_1 -\frac23 a \right) &\text{ if }\veps_1> \frac23a. 
\end{cases} 
\end{equation*}
\end{lem}

\begin{proof}
We define the following three sets $A_{2/3}$, $A_{4/9}$ and $A_{1/3}$:
$$
A_{2/3} = A\cap \left[\frac49+\frac23\veps_1,\frac23\right], \quad 
A_{4/9} = A\cap \left[\frac13+\frac12\veps_1,\frac49+\frac16\veps_1\right], 
$$
and
$$
A_{1/3} = A\cap \left[\frac29+\frac13(a+\veps_1),\frac13 + \frac13 a\right].
$$

By \eqref{a+a}, one gets 
\begin{eqnarray}
\mu_* \left( \frac13 \cdot ( A_1 +A_1 ) \right)  & \geq & \min\left(\mu (A_1),\frac13(\mu (A_1)+{\rm diam}(A_1))\right) \nonumber \\
					 & = & \min\left(  \frac13 -\veps_1 -\veps_2 ,\frac13 \left(  \frac23-2\veps_1 -\veps_2 \right) \right) \nonumber \\
					  & = & \frac29-\frac23\veps_1-\frac13\veps_2 \label{a1a1}
\end{eqnarray}
using the assumption that $\veps_1+2\veps_2\leq1/3$.

Since, by \eqref{hypothese}, the two sets $1/3 \cdot (A_1+A_1)$ and $A_{2/3}$ are disjoint subsets of the interval $\left[ 4/9+ 2 \veps_1 /3, 2/3 \right] $,
one obtains, using \eqref{a1a1},
\begin{equation}
\label{mes_4/9}
\mu (A_{2/3}) \leq \mu \left( \left[\frac49+\frac23\veps_1,\frac23\right] \right) - \mu_* \left( \frac13\cdot(A_1 +A_1 ) \right)\leq \frac{\veps_2}{3}.
\end{equation}

We now prove that
\begin{equation}
\label{mes_1/3}
\mu ( A_{4/9} )  \leq \frac13\veps_2.
\end{equation}
If $A_{4/9}$ has measure zero, 
there is nothing to prove. Thus we denote by $c_1$ the infimum of $A_{4/9}$ and by $c_2$ its supremum, and assume they 
are distinct. We choose a decreasing sequence $(c_1(n))_{n \geq 0} $ in $A_{4/9}$ tending to $c_1$ when $n$ tends to infinity and $< c_2$
(if $c_1$ is in $A_{4/9}$, $c_1 (n)=c_1$ will do).
One has in view of  $1/3+\veps_1/2 \leq c_1 \leq c_1 (n) <  c_2 \leq 4/9+\veps_1/6$,
$$  
\frac13\left(c_1 (n) +\frac23+\veps_1\right)\leq  c_1 (n)  \leq c_2 \leq \frac49+ \frac{\veps_1}{6} \leq \frac{1+c_1}{3} \leq \frac{1+c_1 (n)}{3},
$$
therefore 
\begin{eqnarray*}
A_{4/9} & = & A\cap [c_1,c_2] \\
		& = & (A\cap[c_1,c_1(n)]) \cup (A\cap[c_1(n), c_2]) \\
		& \subset & (A \cap[c_1,c_1(n)])  \cup \left( A \cap \left[\frac13 \left( c_1(n)+ \frac23+\veps_1\right) , \frac13(c_1(n)+1)\right] \right) \\
		& \subset & (A \cap[c_1,c_1(n)])  \cup \left( A \cap\frac13\cdot \left(c_1(n)+\left[\frac23+\veps_1,1\right]\right) \right).
\end{eqnarray*}
Since $c_1 (n) \in A$, , 
assumption \eqref{hypothese} 
implies necessarily that in fact
\begin{equation}
\label{inclusionA49}
A_{4/9} \subset (A\cap[c_1,c_1(n)])\cup \frac13\cdot \left(c_1(n)+ \left(  \left[\frac23+\veps_1,1\right]  \setminus A_1 \right)   \right),
\end{equation}
which in turn gives
\begin{eqnarray*}
\mu ( A_{4/9} ) & \leq & \mu ([c_1,c_1(n)]) + \frac13 \mu \left( \left[\frac23+\veps_1,1\right]  \setminus A_1 \right) \\
			& = &(c_1(n)-c_1)+ \frac13\left(\frac13-\veps_1-\mu (A_1 )\right) \\   
			& = & (c_1(n)-c_1)+ \frac13\veps_2.
\end{eqnarray*}
Letting $n$ tend to infinity in this inequality finishes the proof of \eqref{mes_1/3}.

In the same fashion, if $a\in A$, one obtains
$$
A_{1/3} = 
A \cap \frac13 \cdot \left( a + \left[\frac23+ \veps_1, 1\right] \right) 
= A \cap \frac13 \cdot \left( a + \left(\left[\frac23+ \veps_1, 1\right]\setminus A_1\right) \right)
$$
from which it follows
\begin{equation}
\label{mes_2/9}
\mu (A_{1/3}) \leq   \frac13\veps_2
\end{equation}
and this remains true even if $a\not \in A$ by considering a sequence $(a(n))_{n \geq 0}$ of elements of  $A$ tending to $a$ when $n$ goes to infinity 
arguing similarly as in the proof of \eqref{mes_1/3}.

We now study separately the two inequalities in the statement of the Lemma.
\medskip

\noindent{\underline{First inequality}}. 
\smallskip

Suppose first that $\veps_1\leq 2a/3$, which implies
$$
\sup \left[\frac29+\frac13(a+\veps_1),\frac13 + \frac13a\right] 
\geq \inf \left[\frac13+\frac12\veps_1,\frac49+\frac16\veps_1 \right] ,
$$ 
or, in other words, that $A_{4/9}$ and $A_{1/3}$ overlap. 
One then deduces from  \eqref{mes_4/9}, \eqref{mes_1/3} and \eqref{mes_2/9} that 
\begin{eqnarray*}
\mu \left( A \cap \left[\frac29+\frac{a}{3},1\right]\right)  & \leq & \mu (A_1) + \mu (A_{2/3}) + 
\frac12 \veps_1 + \mu (A_{4/9} ) + \mu (A_{1/3}) + \frac13 \veps_1 \\   
& \leq & \left( \frac13 -\veps_1 -\veps_2 \right)  + \veps_2 + \frac56 \veps_1
= \frac13 - \frac16 \veps_1.
\end{eqnarray*}
And the inequality of Lemma \ref{lem0} follows in this first case.
\medskip

\noindent{\underline{Second inequality}}.
\smallskip

 Until the end of this proof, we assume that $\veps_1> 2a/3$. 

We shall need the sets $B_{1/3}$ and $C_{1/3}$ defined in the following way :
$$
B_{1/3}= A\cap \left[\frac13 + \frac13 a,\frac13+\frac12\veps_1\right], 
\quad 
C_{1/3}= A\cap \left[\frac29+\frac29a,\frac29+\frac13\veps_1\right].
$$
The assumption on the relative sizes of  $\veps_1$ and $a$ shows that
$$
\frac13 \cdot \left(\frac23+\veps_1+B_{1/3}\right)
\subset
\left[\frac13+\frac{a}{9}+\frac{\veps_1}{3},\frac13+\frac{\veps_1}{2}
\right]\subset \left[\frac13+\frac{a}{3},\frac13+\frac{\veps_1}{2}\right].
$$
If $2/3+\veps_1\in A$, \eqref{hypothese} shows that the set on the left 
is intersection-free with $B_{1/3}$ and one thus gets
\begin{equation}
\mu (B_{1/3}) \leq  \mu \left(   \left[\frac13+\frac{a}{3},\frac13+\frac{\veps_1}{2}\right]  \right) - \frac13 \mu  (B_{1/3}) =   \frac{\veps_1}{2}-\frac{a}{3}-\frac{\mu (B_{1/3} )}{3},
\end{equation}
consequently
\begin{equation}
\label{taille_F}
\mu ( B_{1/3} ) \leq\frac34\left(\frac{\veps_1}{2}-\frac{a}{3}\right).
\end{equation}
Once again, the same type of arguments as the ones used to prove \eqref{mes_1/3} shows that this remains true even if $2/3+\veps_1\not\in A$.

Now the inclusion
$$
\frac13 \cdot (B_{1/3}+B_{1/3}) \subset \left[\frac29+\frac{2a}{9},\frac29+\frac{\veps_1}{3}\right]
$$ 
and \eqref{hypothese} show that the set on the left-hand side of this inclusion and $C_{1/3}$ are disjoint and both included in the set on the right-hand side, 
that is, we obtain 
\begin{equation}
\label{FG}
\mu_* \left( \frac13 \cdot (B_{1/3}+B_{1/3}) \right) + \mu (C_{1/3}) \leq \frac{\veps_1}{3} -\frac{2a}{9}.
\end{equation}
By \eqref{BM}, this yields
$$
\frac23 \mu \left(  B_{1/3} \right) + \mu (C_{1/3}) \leq \frac{\veps_1}{3} -\frac{2a}{9}.
$$
Then, using this and \eqref{taille_F}, we derive
\begin{eqnarray*}
\mu (B_{1/3}) +  \mu (C_{1/3}) &   = &  \left( \frac23 \mu \left(  B_{1/3} \right) + \mu (C_{1/3}) \right) + \frac13  \mu \left(  B_{1/3} \right) \\
						& \leq &  \frac{\veps_1}{3} -\frac{2a}{9} + \frac14\left(\frac{\veps_1}{2}-\frac{a}{3}\right)\\
						& = &  \frac{11}{24} \veps_1 -\frac{11}{36} a.
\end{eqnarray*}

We finally deduce from  \eqref{mes_4/9}, \eqref{mes_1/3}, \eqref{mes_2/9}  and the preceding inequality, that 
\begin{eqnarray*}
\mu \left( A \cap \left[\frac29+\frac{a}{3},1\right]\right)  & \leq & \mu (A_1) + \mu (A_{2/3}) + \frac{\veps_1}{2} + \mu (A_{4/9} ) +\mu (B_{1/3})
+ \mu (A_{1/3}) \\
&&\hspace{7.68cm} + \frac{a}{3}+\mu (C_{1/3})\\   
& \leq & \left( \frac13 -\veps_1 -\veps_2 \right)  + \veps_2 + \frac12 \veps_1+ \frac13 a + \frac{11}{24} \veps_1 -\frac{11}{36} a\\
& = & \frac13 -\frac{1}{24} \veps_1 +\frac{1}{36} a\\
& = & \frac13 -\frac{1}{24} \left( \veps_1 -\frac23 a \right).
\end{eqnarray*}
Hence the announced inequality.
\end{proof}

The second central lemma, needed for the proof of Theorem \ref{thm2}, deals with attaining the bound $1/3$ in Lemma \ref{lem0}. 
Here it is.

\begin{lem}
\label{lem1}
Let $A$ be a measurable 3-sum-free subset of $[0,1]$ 
such that $\sup{A}=1$. We define
$$
a =\inf{A} , \quad  A_{1}=A\cap\left[\frac23,1\right], \quad
\veps_1 = \inf{A_{1}} - \frac23 \quad  \text{ and } \quad \veps_2= \left( \frac13-\veps_1 \right)- \mu ( A_1 ).
$$
We assume $\veps_1 +2 \veps_2 \leq 1/3$ and $a>0$. Then, $\mu \left( A \cap \left[ 2/9+ a/3,1\right]\right)  = 1/3$ implies
$$
\veps_1 = \veps_2=0.
$$
 \end{lem}

\begin{proof}
In this proof we will use freely the notation introduced in the preceding lemma. 

We first apply Lemma \ref{lem0} to $A$. The precise inequality obtained there implies that we cannot have  $\veps_1 >2a/3$, 
since we would get $\mu ( A \cap [2/9+a/3,1] )  < 1/3$. Thus $\mu \left( A \cap \left[2/9+a/3,1\right]\right)  = 1/3$
implies $\veps_1 \leq 2a/3$ and then $\veps_1 =0$ in view of the precise formula in this case.

We now turn to the core of this proof and show that
\begin{equation}
\label{eps2nul}
\veps_2=0
\end{equation}
and assume for a contradiction that $\veps_2 \neq 0$. 

Following carefully the proof of  Lemma \ref{lem0} shows that $\mu \left( A \cap \left[ 2/9+ a/3,1\right]\right)  = 1/3$ 
implies
\begin{equation}
\label{lesmu}
\mu(A_{1/3})=\mu(A_{4/9})=\mu(A_{2/3})=\frac{\veps_2}{3},\quad \text{ and }\quad \mu \left( \left[ \frac13, \frac13 + \frac{a}{3} \right] \cap A \right)=0,
\end{equation}
this last equality being tantamount to saying that the intersection of $A_{1/3}$ and $A_{4/9}$ has measure zero.

The function $f:\left(x\mapsto\mu([1/3,x]\cap A_{4/9})\right)$ is a non-decreasing non-negative continuous function on $[1/3,4/9]$ 
such that $f$ is identically $0$ on $[1/3,(1+a)/3]$ and $f(4/9)=\veps_2/3 >0$. We define $\tilde{c_1}$ as the following infimum 
$$
\tilde{c_1} = \inf \{ x \in [1/3,4/9], f(x)>0\}.
$$ 
We have 
\begin{equation}
\label{c1t}
\tilde{c_1}\geq \frac{1+a}{3}.
\end{equation} 
Furthermore, $\mu([1/3,x]\cap A_{4/9})=0$ for any $x\in[1/3,\tilde{c_1}]$, whereas $ \mu\left( A_{4/9}\cap \left[ \tilde{c_1}, \tilde{c_1}+\eta\right] \right) >0$ 
for any $\eta>0$.

We choose a real number $\eta$ such that $0 < \eta < \min(a,\veps_2)/3$. 
Let $v$ be any element of $[\tilde{c_1}, \tilde{c_1} +\eta]\cap A_{4/9}$. 
We have, using \eqref{c1t}, 
$$
\frac13 <  v<   \tilde{c_1} + \eta <  \tilde{c_1} + \frac{a}{3} \leq  \tilde{c_1} + \left( \tilde{c_1}- \frac13 \right) <  \tilde{c_1} + 2\left( \tilde{c_1}- \frac13 \right)
= 3 \tilde{c_1}- \frac23
$$
from which it follows that
$$
\frac13 \left( v+\frac23 \right) \leq \tilde{c_1}\leq c_2\leq\frac49\leq\frac13(v+1),
$$
on recalling that $c_2 =\sup A_{4/9}$.
Going back to the proof that $\mu(A_{4/9})= \veps_2 /3$ in Lemma \ref{lem0}, the preceding inequalities allow 
to obtain
$$
A_{4/9}  \subset (A\cap[c_1,\tilde{c_1}])\cup  (A\cap[\tilde{c_1}, c_2])
\subset (A\cap[c_1,\tilde{c_1}])\cup \frac13 \cdot \left(A\cap\left( v+  \left[\frac23,1\right]   \right)   \right).
$$
As previously, assumption \eqref{hypothese}  yields
$$
A_{4/9} \subset (A\cap[c_1,\tilde{c_1}])\cup \frac13 \cdot \left( v+ \left(  \left[\frac23,1\right]  \setminus A_1 \right)   \right).
$$

But the sets $A_{4/9}$ and $1/3 \cdot \left( v+ \left(  \left[2/3,1\right]  \setminus A_1 \right)   \right)$ have the same measure $\veps_2/3$
while $A\cap[c_1,\tilde{c_1}]$ has measure zero. We therefore deduce that, for any $v$ in $[\tilde{c_1}, \tilde{c_1} +\eta]\cap A_{4/9}$,
\begin{equation}
\label{uu'}
A_{4/9} = \frac13 \cdot \left( v+ \left(  \left[\frac23,1\right]  \setminus A_1 \right)   \right)
\end{equation}
up to a set of measure zero. Choosing $u\not=u' $  in  $[\tilde{c_1}, \tilde{c_1} +\eta]\cap A_{4/9}$
(such $u$ and $u'$ do exist since $\mu\left([\tilde{c_1}, \tilde{c_1} +\eta]\cap A\right)\not=0$), and applying \eqref{uu'}
consecutively to $v=u$ and $v=u'$, we get, up to sets of measure zero, 
$$
u+ \left(  \left[\frac23,1\right]  \setminus A_1 \right)    =  u'+ \left(  \left[\frac23,1\right]  \setminus A_1 \right).
$$ 
This implies that 
$$
\veps_2= 3 \mu\left(  \left[\frac23,1\right]  \setminus A_1 \right) =0,
$$ 
a contradiction. Assertion \eqref{eps2nul} is therefore proved.
\end{proof}

\section{Proof of Theorem \ref{main}}

Let us begin with a simple consequence of Lemma \ref{lem0}.

\begin{lem}
\label{lemrec}
Let $A$ be a measurable 3-sum-free subset of $[0,1]$ 
such that $\sup{A}=1$ and $\mu (A) \geq 5/12$. Then, 
$$
\mu \left( A \cap \left[ \frac{\inf A }{3}+ \frac29,1 \right] \right) \leq \frac13.
$$
\end{lem}

\begin{proof}
Define $\veps_1$ and $\veps_2$ as in the statement of Lemma \ref{lem0}.
The assumptions and Lemma \ref{lemMR} show that
$$
\frac{5}{12} \leq \mu (A) \leq \frac13 + \frac12 \mu \left( A \cap \left[\frac23,1\right] \right)
= \frac12 (1- \veps_1 - \veps_2),
$$
or, equivalently, $\veps_1 + \veps_2 \leq 1/6$, which implies
$\veps_1+2\veps_2\leq 2(\veps_1+\veps_2)\leq1/3$ and makes it possible to apply Lemma \ref{lem0},
which in turns concludes the proof.
\end{proof}

We can now prove Theorem \ref{main}, the main result of this paper.   
   
 \begin{proof} [Proof of Theorem \ref{main}]
We may, without loss of generality, assume that $\sup(A)=1$, since otherwise, we consider $(1/\sup(A)) \cdot A$.  
Since $77/177 >5/12$, we may also assume that $\mu ( A) \geq 5/12$, otherwise there is nothing to prove. 
 Therefore, applying Lemma \ref{lemrec}, we get
 \begin{equation}
 \label{muA_R}
 \mu (A) \leq \frac13+ \mu (R) \quad \text{ where } \quad a= \inf A \quad \text{ and  } \quad R=A\cap \left[a,\frac29+\frac13a\right].
 \end{equation}
Notice that
\begin{equation}
\label{muR}
\mu (R)  \leq  \mu \left(  \left[ a,\frac29+\frac13 a \right]  \right) = \frac29 -\frac{2a}{3}.
\end{equation}

Since $R$ is non-empty (its measure is at least $1/12$, by assumption), we define 
$$
r=\sup R, \quad R'= \frac{1}{r} \cdot R \quad \text{ and } \quad R'_1 = R' \cap \left[ \frac23,1\right] = \frac{1}{r} \cdot \left(R\cap\left[\frac23 r,r\right]\right)
$$ 
and put 
$$
\eta_1= \inf{R'_1} - \frac23, \quad \eta_2=\frac13-\eta_1-\mu(R'_1). $$ 
We distinguish two cases.
 \medskip
 
\noindent \underline{Case 1}:  $\eta_1+2\eta_2 \leq 1/3$. 
\smallskip

We apply Lemma \ref{lem0} to the set $R'$ and get
 $$
 \mu \left( R'  \right)  \leq \frac13+\mu \left( R' \cap \left[\frac{a}{r},\frac29+\frac{a}{3r}\right]\right).
 $$
This implies 
 \begin{equation}
 \label{blu}
 \mu \left( R  \right)  \leq \frac{r}{3}+\mu \left( R \cap \left[ a,\frac{2r}{9}+\frac{a}{3}\right]\right) = \frac{r}{3}+\mu \left( R_0 \right),
 \end{equation}
if we denote  $$
R_0= R \cap \left[a,\frac{2r}{9} +\frac{a}{3}\right].
$$

If $R_0=\emptyset$, then by \eqref{muR}, \eqref{blu} and the inequality $r \leq 2/9 +a/3$, we obtain
$$
\mu(R) \leq \min \left( \frac{r}{3},\frac29-\frac{2a}{3} \right)\leq\min\left(\frac{2}{27}+ \frac{a}{9}, \frac29-\frac{2a}{3} \right)
\leq \frac{2}{21},
$$ 
this maximum value being attained for $a=4/21$. Thus, in this case we must have 
$$
\mu(A) \leq \frac13+\frac{2}{21}=\frac37< \frac{77}{177}
$$
and we are done.

From now on, we therefore assume that $R_0$ is a non empty set and we define $b=\sup R_0$.
Applying Lemma \ref{majreste} to $R_0$ together with the obvious inequality $\mu (R_0) \leq b-a$ yields
\begin{align*}
\mu (R_0 )  &\leq\min\left( \frac{2 b -a}{4},b-a\right)\\
&\leq  \min\left(\frac{r}{9}-\frac{1}{12}a,\frac{2r}{9}-\frac23a\right)\\
\end{align*}
since $b \leq 2r/9 +a/3$.
Therefore, by \eqref{blu}, we have
\begin{align*}
\mu (R)  &\leq  \min\left(\frac{4r}{9}-\frac{1}{12}a,\frac{5r}{9}-\frac23a\right).
\end{align*}
Using \eqref{muA_R} and $r \leq 2/9+a/3$, we get 
\begin{align*}
\mu (A ) &\leq\frac13+\min\left(\frac49\left(\frac29+\frac{a}{3}\right)-\frac{1}{12}a,\frac59\left(\frac29+\frac13a\right)-\frac23a\right)\\
&\leq\frac13 + \min\left(\frac{8}{81}+\frac{7}{108}a,\frac{10}{81}-\frac{13}{27}a\right)\\
&\leq  \frac{77}{177},
\end{align*}
this value being attained uniquely in $a=8/177$. 
\medskip

\noindent \underline{Case 2} : Assume now that $\eta_1 + 2\eta_2 >1/3$. 
\smallskip

In particular, $\eta_1 +\eta_2 > 1/6$. The assumptions and Lemma \ref{lemMR} give $\mu ( R' )\leq {5}/{12}$, thus 
\begin{equation}
\label{majR1}
\mu (R ) \leq \frac{5r}{12} .
\end{equation}

We now prove that
\begin{equation}
\label{labelle}
\mu (R) \leq \max\left(\frac12(r-a), \frac{2r-a}{6} \right).  
\end{equation}
Indeed, if $\mu (R) > (r-a)/2={\rm diam}(R)/2$, then \eqref{a+a} implies
$$
\mu_* (R+ R ) \geq \mu (R )+{\rm diam}(R)=\mu(R)+(r-a).
$$
Since $1/3 \cdot ( R+R ) \subset \left[ 2a/3, 2r/3 \right]$ and $( 1/3 \cdot (R+R ) ) \cap R=\emptyset$, 
we get
$$
\mu \left( R \cap\left[a, \frac23 r \right]  \right)=   \mu \left( R \cap\left[\frac23 a, \frac23 r \right]  \right)    \leq \frac23(r-a)- \mu_* \left(  \frac13\cdot (R+R ) \right) \leq \frac13(r-a)-\frac13 \mu (R).
$$
It follows that
\begin{eqnarray*}
\mu (R ) & = & \mu \left( R \cap\left[a, \frac23 r \right]  \right) + \mu \left( R \cap\left[ \frac{2}{3}r, r \right] \right)\\
		& \leq &  \frac13(r-a)-\frac13 \mu (R) + \frac{r}{3} - (\eta_1 + \eta_2 )r \\
		& \leq & \frac13(r-a)-\frac16 (r-a) + \frac{r}{6}\\
		& = & \frac{2r-a}{6}
\end{eqnarray*}
and assertion \eqref{labelle} is proved.

Synthetizing \eqref{majR1}, \eqref{muA_R} and \eqref{labelle}, we finally obtain
$$
\mu (A)	\leq \frac13+\min\left( \max \left( \frac12(r-a), \frac13r- \frac16a \right) ,\frac{5}{12}r\right).
$$
Taking into account $r\leq 2/9+a/3$, we get
$$
\mu (A)	\leq  \frac13+\min\left(  \max \left( \frac19-\frac13 a, \frac{2}{27}-\frac1{18}a \right) ,\frac{5}{54}+\frac{5}{36}a\right)
.$$
If $a<2/15$, this yields
$$
\mu (A) \leq \frac13+\min\left(  \frac19-\frac13 a,\frac{5}{54}+\frac{5}{36}a\right) = \frac{22}{51},
$$
this maximum being attained uniquely in $a=2/51$.
If $a \geq 2/15$,we have
$$
\mu (A) \leq \frac13+\min\left(   \frac{2}{27}-\frac1{18} a  ,\frac{5}{54}+\frac{5}{36}a\right) = \frac13+ \left( \frac{2}{27}-\frac1{18} a \right)
\leq \frac13 + \frac{2}{27}=\frac{11}{27}.
$$
Since both $22/51$ and $11/27$ are $< 77/177$, we obtain, in this second case, that $\mu (A) < 77/177$.

This concludes the proof of Theorem \ref{main}.

\end{proof}

\section{The inverse result: proof of Theorem \ref{thm2}}

This section is devoted to the proof of the structural characterization of 3-sum-free sets with 
maximal measure. We start with a lemma which contains the core of the structural result.

\begin{lem}
\label{equality}
Let $A$ be a measurable 3-sum-free subset of $[0,1]$  
satisfying $\mu(A)=77/177$.
Then $\mu(A\Delta \Acal_0)=0$ where $A\Delta \Acal_0$ stands for the symmetric difference between $A$ 
and $\Acal_0$ as defined in formula  \eqref{ensA}.
\end{lem}

\begin{proof}
Let us assume that we have a set $A\subset[0,1]$ with no solution to the equation $x+y=3z$ such that $\mu(A)=77/177$. 
We can assume that $\sup(A)= 1$, otherwise $(1 /\sup{A}) \cdot A$ would contradict Theorem \ref{main}.

For the sake of clarity, we recall the notation we shall use in this proof, namely
$$
a=\inf A,\quad A_{1}=A\cap\left[\frac23,1\right],\quad  \veps_1 = \inf{A_{1}} - \frac23, \quad \veps_2= \frac13 - \veps_1 - \mu (A_{1}),
$$
$$
R=A\cap \left[a,\frac29+\frac13a\right],\quad r=\sup R, \quad R'_1 =   \left( \frac{1}{r} \cdot R \right) \cap \left[ \frac23,1\right] ,
$$
and
$$
R_0= R \cap \left[a,\frac29 r+\frac{a}{3}\right],\quad b = \sup R_0.
$$

If we examine the proof of Theorem \ref{main}, we notice first that we must have $\mu (A)= 1/3 + \mu (R)$ that is,
\begin{equation}
\label{casdeg}
\mu \left( A \cap \left[ \frac29 + \frac{a}{3}, 1 \right] \right) = \mu (A \setminus R) = \frac13.
\end{equation}

Furthermore,
we cannot be in Case 2 of the proof of Theorem \ref{main} since the conclusion is  then that $\mu (A) \leq 22/51 < 77/177$. 
Therefore we must be in Case 1 (and more precisely subcase $R_0 \neq \emptyset$) and several inequalities occurring in the course of 
the proof must actually be equalities. In particular, we must have 
$$
a=\frac{8}{177}, \quad r=\frac{2}{9}+\frac{a}{3}=\frac{14}{59},\quad \text{ and }  b=\frac{2r}{9}+\frac{a}{3}= \frac{4}{59}.
$$

Now, since $a>0$, Lemma \ref{lem1} shows that \eqref{casdeg} implies 
$
\veps_1 = \veps_2 = 0,
$
thus $\mu (A_{1}) = 1/3=\mu (A \setminus R)$, therefore, up to a set of measure zero, we have
\begin{equation}
\label{decompoA}
A=R\cup \left( \frac23, 1 \right).
\end{equation}

Moreover, in the course of the proof of Theorem \ref{main} we also applied Lemma \ref{lem0} to $(1/r)\cdot R$, so,
in the equality case, similar arguments as above yield, up to a set of measure zero,
$$
R\cap[2r/9+a/3,r]=(2r/3,r) = (28/177, 14/59).
$$
What remains of $A$ is, by definition, contained in $[a,b]$. It follows that up to a set of measure zero
$$
A \subset \left(\frac{8}{177}, \frac{4}{59}\right)\cup  \left(\frac{28}{177}, \frac{14}{59}\right)\cup  \left(\frac{2}{3}, 1\right)=\Acal_0.
$$
This implies the statement of the lemma since $A$ and $\Acal_0$ have the same measure.
\end{proof}

Before coming to the proof of our inverse theorem, we recall a kind of prehistorical lemma in our context.

\begin{lem}
\label{prehisto}
Let $X$ and $Y$ be two subsets of $\R$. Let $\alpha, \beta, \gamma, \delta \in \R$ such that 
$X \subset (\alpha, \beta )$, $\mu_* (X) = \beta -\alpha$, $Y \subset (\gamma, \delta )$, $\mu_* (Y) = \delta -\gamma$,
then
$$
X+Y= (\alpha + \gamma,  \beta+ \delta).
$$
\end{lem}

\begin{proof}
Let $v \in (\alpha + \gamma,  \beta+ \delta)$. It can be written as $v=\phi + \chi$ with 
$\phi \in (\alpha, \beta )$ and $\chi \in (\gamma, \delta )$. Let
$$
\theta =\frac12  \min (|\phi- \alpha |, |\phi-\beta |, |\chi-\gamma |, |\chi-\delta |).
$$ 
It follows that
$(\phi-\theta, \phi+\theta) \subset (\alpha, \beta )$ and $(\chi-\theta, \chi+\theta) \subset (\gamma, \delta )$.
Since $X$ and $Y$ are of full measure, we must have $\mu_* (X \cap (\phi-\theta, \phi+\theta))=
\mu_* (Y \cap (\chi-\theta, \chi+\theta)) = 2 \theta$.
Moreover, the set $v - (\phi-\theta, \phi+\theta) = (\chi-\theta, \chi+\theta)$ and it follows that
we must also have 
$$
\mu_* ( (v - X) \cap (\chi-\theta, \chi+\theta)   )= \mu_* ( Y  \cap (\chi-\theta, \chi+\theta)) = 2 \theta.
$$ 
Consequently the two full-measure 
in $(\chi-\theta, \chi+\theta)$ sets $((v-X) \cap (\chi-\theta, \chi+\theta))$ and $Y \cap (\chi-\theta, \chi+\theta))$ 
must intersect which shows that there are a $x$ in $X$ and a $y$ in $Y$ such that $v-x=y$ or $v=x+y\in X+Y$. 
Hence the result, this being valid for any $v$. 
\end{proof}

Here is a lemma generalizing \eqref{hypothese} which will be key in the proof of Theorem \ref{thm2}. Its proof is immediate.

\begin{lem}
\label{caract}
Let $A$ be a $3$-sum-free set. 
Then
$$
\left( \frac{1}{3} \cdot (A+A) \right) \cap A = \emptyset\quad \text{ and }\quad ((3 \cdot A)-A) \cap A = \emptyset.
$$
\end{lem}

We are now ready to conclude the proof of the Chung-Goldwasser-Matolcsi-Ruzsa conjecture and prove Theorem \ref{thm2}.

\begin{proof}[Proof of Theorem \ref{thm2}]
Applying Lemma \ref{equality} gives that, under the hypothesis of the theorem, $\mu(A\Delta \Acal_0)=0$. Equivalently, 
$A$ is of the form
$$
A =  U \cup V \cup A_1 \cup Z
$$
with 
$$
U \subset  \left[\frac{8}{177}, \frac{4}{59}\right],\quad V \subset  \left[\frac{28}{177}, \frac{14}{59}\right],\quad  
\text{ and }\quad  A_1 \subset \left[\frac{2}{3}, 1\right],
$$
these three sets being of maximal measure in their respective intervals; and $\mu (Z)=0$.

Having noticed that if a set is of full measure in a given interval then dilating it by a constant factor 
transforms it as a full measure set in the dilated interval, an easy computation, based on Lemma \ref{prehisto}, shows that
$$
 (3 \cdot V -V)  \cup \left( \frac13 \cdot (A_1 + A_1 ) \right) =  \left(\frac{14}{59},      \frac{98}{177}  \right) \cup \left(\frac{4}{9}, \frac{2}{3}\right) = 
 \left(\frac{14}{59}, \frac{2}{3}\right).
$$
In the same way, we compute that 
$$
(3 \cdot U) - U = \left(\frac{4}{59},      \frac{28}{177}  \right)
$$
and
$$
(3 \cdot V) - A_1 = \left(-\frac{31}{59},      \frac{8}{177}  \right).
$$
By Lemma \ref{caract}, the union of all these sets is intersection-free with $A$, therefore $A$ is contained in its complementary 
set in $[0,1]$, namely
$$
A \subset \left[\frac{8}{177}, \frac{4}{59} \right] \cup \left[ \frac{28}{177} , \frac{14}{59} \right] \cup \left[ \frac{2}{3}, 1 \right].
$$
It follows that $Z= \emptyset$.

Studying the different cases with the endpoints leads to the result.
\end{proof}

\bigskip\bigskip

\end{document}